\documentclass[12pt]{article}
\usepackage{amsmath,amssymb,amsfonts}
\usepackage{epsfig}

\newtheorem{theorem}{Theorem}[section]
\newtheorem{lemma}[theorem]{Lemma}

\newtheorem{claim}{Claim}

\newenvironment{proof}{\noindent\emph{Proof.}\hspace{.25em}}{\hspace*{\fill}
$\Box$\newline}

\def\pfsp{\hskip 1em}

\newcommand{\skpf}[1]{\noindent{\it Sketch of Proof\,:}\pfsp #1 {\hfill $\Box$} \smallskip}
\newcommand{\pfcl}[1]{\noindent{\it Proof\,:}\pfsp #1 \QED \smallskip}
\newcommand{\pff}[2]{\noindent{\it Proof~of~#1\,:}\pfsp #2 {\hfill $\Box$}}
\def \QED {\hfill $\triangle$}      % end of proof
\def \QD1 {\hfill $\spadesuit$}

\newcommand{\set}[2]{\{#1 \;|\; #2 \}}
\newcommand{\ems}{\varnothing}
\newcommand{\sm}{\setminus}

\newcommand{\De}{\Delta}
\newcommand{\de}{\delta}
\newcommand{\col}{{\rm col}}
\newcommand{\mad}{{\rm mad}}
\newcommand{\la}{\lambda}
\newcommand{\ka}{\kappa}
\newcommand{\om}{\omega}

\newcommand{\cn}{\chi}

\newcommand{\bG}{\partial_G}
\newcommand{\bH}{\partial_H}

\newcommand{\cH}{{\cal H}}
\newcommand{\cB}{{\cal B}}
\newcommand{\cC}{{\cal C}}
\newcommand{\cP}{{\cal P}}
\newcommand{\cL}{{\cal L}}

\newcommand{\CO}{{\cal CO}}

\numberwithin{equation}{section}

\begin{document}
%\title{\bf Coloring graphs with bounded local edge connectivity}
\title{\bf A Brooks type theorem for the maximum local edge connectivity}

\author{{{Michael Stiebitz}
\thanks{The authors thank the Danish Research Council for support through the program Algodisc.}
\thanks{
Technische Universit\"at Ilmenau, Inst. of Math., PF 100565, D-98684 Ilmenau, Germany. E-mail
address: michael.stiebitz@tu-ilmenau.de}}
\and{{Bjarne Toft}\footnotemark[1]~\thanks{University of Southern Denmark, IMADA, Campusvej 55, DK-5320 Odense M, Denmark E-mail
address: btoft@imada.sdu.dk} }}
\date{}
\maketitle
\begin{abstract}
For a graph $G$, let $\cn(G)$ and $\la(G)$ denote the chromatic number of $G$ and the maximum local edge connectivity of $G$, respectively. A result of Dirac \cite{Dirac53} implies that every graph $G$ satisfies $\cn(G)\leq \la(G)+1$. In this paper we characterize the graphs $G$ for which $\cn(G)=\la(G)+1$. The case $\la(G)=3$ was already solved by Alboulker {\em et al.\,} \cite{AlboukerV2016}. We show that a graph $G$ with $\la(G)=k\geq 4$ satisfies $\cn(G)=k+1$ if and only if $G$ contains a block which can be obtained from copies of $K_{k+1}$ by repeated applications of the Haj\'os join.
\end{abstract}



\section{Introduction and main result}

The paper deals with the classical vertex coloring problem for graphs. The term graph refers to a finite undirected graph without loops and without multiple edges. The {\em chromatic number} of a graph $G$, denoted by $\cn(G)$, is the least number of colors needed to color the vertices of $G$ such that each vertex receives a color and adjacent vertices receive different colors. There are several degree bounds for the chromatic number. For a graph $G$, let $\de(G)=\min_{v\in V(G)}d_G(v)$ and $\De(G)=\max_{v\in V(G)}d_G(v)$ denote the {\em minimum degree} and the {\em maximum degree} of $G$, respectively. Furthermore, let $$\col(G)=1+\max_{H\subseteq G}\de(H)$$ denote the {\em coloring number} of $G$, and let $$\mad(G)=\max_{\ems\not=H\subseteq G} \frac{2|E(H)|}{|V(H)|}$$ denote the {\em maximum average degree} of $G$. By $H\subseteq G$ we mean that $H$ is a subgraph of $G$. If $G$ is the {\em empty graph}, that is, $V(G)=\ems$, we briefly write $G=\ems$ and define $\de(G)=\De(G)=\mad(G)=0$ and $\col(G)=1$. A simple sequential coloring argument shows that $\cn(G)\leq \col(G)$, which implies  that every graph $G$ satisfies
$$\cn(G)\leq \col(G)\leq \lfloor \mad(G)\rfloor+1\leq \De(G)+1.$$
These inequalities were discussed in a paper by Jensen and Toft \cite{JensenT95}.
Brooks' famous theorem provides a characterization for the class of graphs $G$ satisfying $\cn(G)=\De(G)+1$. Let $k\geq 0$ be an integer. For $k\not=2$, let $\cB_k$ denote the class of complete graphs having order $k+1$; and let $\cB_2$ denote the class of odd cycles. A graph in $\cB_k$ has maximum degree $k$ and chromatic number $k+1$. Brooks' theorem \cite{Brooks41} is as follows.

\begin{theorem} [Brooks 1941]
Let $G$ be a non-empty graph. Then $\cn(G)\leq \De(G)+1$ and equality holds if and only if $G$ has a connected component belonging to the class $\cB_{\De(G)}$.
\label{Th:Brooks}
\end{theorem}

In this paper we are interested in connectivity parameters of graphs. Let $G$ be a graph with at least two vertices. The {\em local connectivity} $\ka_G(v,w)$ of distinct vertices $v$ and $w$ is the maximum number of internally vertex disjoint $v$-$w$ paths of $G$. The {\em local edge connectivity} $\la_G(v,w)$ of distinct vertices $v$ and $w$ is the maximum number of edge-disjoint $v$-$w$ paths of $G$. The {\em maximum local connectivity} of $G$ is
$$\ka(G)=\max\set{\ka_G(v,w)}{v,w\in V(G), v\not=w},$$
and the {\em maximum local edge connectivity} of $G$ is
$$\la(G)=\max\set{\la_G(v,w)}{v,w\in V(G), v\not=w}.$$
For a graph $G$ having only one vertex, we define $\ka(G)=\la(G)=0$. Clearly, the definition implies that $\ka(G)\leq \la(G)$ for every graph $G$. By a result of Mader \cite{Mader73} it follows that $\de(G)\leq \ka(G)$. Since  $\ka$ is a monotone graph parameter in the sense that $H\subseteq G$ implies $\ka(H)\leq \ka(G)$, it follows that every graph $G$ satisfies $\col(G)\leq \ka(G)+1$. Consequently, every graph $G$ satisfies
\begin{equation}
\label{Equ:lambda}
\cn(G)\leq \col(G) \leq \ka(G)+1\leq \la(G)+1\leq \De(G)+1.
\end{equation}
Our aim is to characterize the class of graphs $G$ for which $\cn(G)=\la(G)+1$. For such a characterization we use the fact that if we have an optimal coloring of each block of a graph $G$, then we can combine these colorings to an optimal coloring of $G$ by permuting colors in the blocks if necessary. For every non-empty graph $G$, we thus have
\begin{equation}
\label{Equ:Block}
\cn(G)=\max\set{\cn(H)}{H \mbox{ is a block of } G}.
\end{equation}
We also need a famous construction, first used by Haj\'os \cite{Hajos61}. Let $G_1$ and $G_2$ be two vertex-disjoint graphs and, for $i=1,2$, let $e_i=v_iw_i$ be an edge of $G_i$. Let $G$ be the graph obtained from $G_1$ and $G_2$ by deleting the  edges $e_1$  and $e_2$ from $G_1$ and $G_2$, respectively, identifying the vertices $v_1$ and $v_2$, and adding the new edge $w_1w_2$. We then say that $G$ is the {\em Haj\'os join} of $G_1$ and $G_2$ and write $G=(G_1,v_1,w_1)\bigtriangleup (G_2,v_2,w_2)$ or briefly $G=G_1 \bigtriangleup G_2$.

For an integer $k\geq 0$ we define a class $\cH_k$ of graphs as follows. If $k\leq 2$, then $\cH_k=\cB_k$. The class $\cH_3$ is the smallest class of graphs that contains all odd wheels and is closed under taking Haj\'os joins. Recall that an {\em odd wheel} is a graph obtained from on odd cycle by adding a new vertex and joining this vertex to all vertices of the cycle. If $k\geq 4$, then $\cH_k$ is the smallest class of graphs that contains all complete graphs of order $k+1$ and is closed under taking Haj\'os joins. Our main result is the following counterpart of Brooks' theorem. In fact, Brooks' theorem may easily be deduced from it.

\begin{theorem}
Let $G$ be a non-empty graph. Then $\cn(G)\leq \la(G)+1$ and equality holds if and only if $G$ has a block belonging to the class $\cH_{\la(G)}$.
\label{Th:local}
\end{theorem}

For the proof of this result, let $G$ be a non-empty graph with $\la(G)=k$. By (\ref{Equ:lambda}), we obtain $\cn(G)\leq k+1$. By an observation of Haj\'os \cite{Hajos61} it follows that every graph in $\cH_k$ has chromatic number $k+1$. Hence if some block of $G$ belongs to $\cH_k$, then (\ref{Equ:Block}) implies that $\cn(G)=k+1$. So it only remains to show that if $\cn(G)=k+1$, then some block of $G$ belongs to $\cH_k$. For proving this, we shall use the critical graph method, see \cite{StiebitzT2015}.

A graph $G$ is {\em critical} if every proper subgraph $H$ of $G$ satisfies $\cn(H)<\cn(G)$. We shall use the following two properties of critical graphs.
As an immediate consequence of (\ref{Equ:Block}) we obtain that if $G$ is a critical graph, then $G=\ems$ or $G$ contains no separating vertex, implying that $G$ is its only block. Furthermore, every graph contains a critical subgraph with the same chromatic number.

Let $G$ be a non-empty graph with $\la(G)=k$ and $\cn(G)=k+1$. Then $G$ contains a critical subgraph $H$ with chromatic number $k+1$, and we obtain that $\la(H)\leq \la(G)=k$. So the proof of Theorem~\ref{Th:local} is complete if we can show that $H$ is a block of $G$ which belongs to $\cH_k$. For an integer $k\geq 0$, let $\cC_k$ denote the class of graphs $H$ such that $H$ is a critical graph with chromatic number $k+1$ and with $\la(H)\leq k$. We shall prove that the two classes $\cC_k$ and $\cH_k$ are the same.

\section{Connectivity of critical graphs}

In this section we shall review known results about the structure of critical graphs. First we need some notation. Let $G$ be an arbitrary graph. For an integer $k\geq 0$, let $\CO_k(G)$ denote the set of all colorings of $G$ with color set $\{1,2, \ldots, k\}$. Then a function $f:V(G) \to \{1,2, \ldots, k\}$ belongs to $\CO_k(G)$ if and only if $f^{-1}(c)$ is an independent vertex set of $G$ (possibly empty) for every color $c\in \{1,2, \ldots, k\}$. A set $S\subseteq V(G) \cup E(G)$ is called a {\em separating set} of $G$ if $G-S$ has more components than $G$. A vertex $v$ of $G$ is called a {\em separating vertex} of $G$ if $\{v\}$ is a separating  set of $G$. An edge $e$ of $G$ is called a {\em bridge} of $G$ if $\{e\}$ is a separating set of $G$. For a vertex set $X\subseteq V(G)$, let $\bG(X)$ denote the set of all edges of $G$ having exactly one end in $X$. Clearly, if $G$ is connected and $\ems\not= X \varsubsetneq V(G)$, then $F=\bG(X)$ is a separating set of edges of $G$. The converse is not true. However if $F$ is a minimal separating edge set of a connected graph $G$, then $F=\bG(X)$ for some vertex set $X$. As a consequence of Menger's theorem about edge connectivity, we obtain that if $v$ and $w$ are two distinct vertices of $G$, then
$$\la_G(v,w)=\min\set{|\bG(X)|}{X\subseteq V(G), v\in X, w\not\in X}.$$

Color critical graphs were first introduced and investigated by Dirac in the 1950s. He established the basic properties of critical graphs in a series of papers \cite{Dirac52},  \cite{Dirac53} and \cite{Dirac57}. Some of these basic properties are listed in the next theorem.

\begin{theorem} [Dirac 1952]
Let $G$ be a critical graph with chromatic number $k+1$ for an integer $k\geq 0$. Then the following statements hold:
\begin{itemize}

\item[{\rm (a)}] $\de(G)\leq k$
\item[{\rm (b)}] If $k=0,1$, then $G$ is a complete graph of order $k+1$; and if $k=2$, then $G$ is an odd cycle.
\item[{\rm (c)}] No separating vertex set of $G$ is a clique of $G$. As a consequence, $G$ is connected and has no separating vertex, i.e., $G$ is a block.
\item[{\rm (d)}] If $v$ and $w$ are two distinct vertices of $G$, then $\la_G(v,w)\geq k$. As a consequence $G$ is $k$-edge-connected.
\end{itemize}
\label{Th:Dirac}
\end{theorem}

Theorem~\ref{Th:Dirac}(a) leads to a very natural way of classifying the vertices of a critical graph into two classes. Let $G$ be a critical graph with chromatic number $k+1$. The vertices of $G$ having degree $k$ in $G$ are called {\em low vertices} of $G$, and the remaining vertices are called {\em high vertices} of $G$. So any high vertex of $G$ has degree at least $k+1$ in $G$. Furthermore, let $G_L$ be the subgraph of $G$ induced by the low vertices of $G$, and let $G_H$ be the subgraph of $G$ induced by the high vertices of $G$. We call $G_L$ the {\em low vertex subgraph} of $G$ and $G_H$ the {\em high vertex subgraph} of $G$. This classification is due to Gallai \cite{Gallai63a} who proved the following theorem. Note that statements (b) and (c) of Gallai's theorem are simple consequences of statement (a), which is an extension of Brooks' theorem.

\begin{theorem} [Gallai 1963]
Let $G$ be a critical graph with chromatic number $k+1$ for an integer $k\geq 1$. Then the following statements hold:
\begin{itemize}
\item[{\rm (a)}] Every block of $G_L$ is a complete graph or an odd cycle
\item[{\rm (b)}] If $G_H=\ems$, then $G$ is a complete graph of order $k+1$ if $k\not=2$, and $G$ is an odd cycle if $k=2$.
\item[{\rm (c)}] If $|V(G_H)|=1$, then either $G$ has a separating vertex set of two vertices or $k=3$ and $G$ is an odd wheel.
\end{itemize}
\label{Th:Gallai}
\end{theorem}

As observed by Dirac, a critical graph is connected and contains no separating vertex. Dirac \cite{Dirac52} and Gallai \cite{Gallai63a} characterized critical graphs
having a separating vertex set of size two. In particular, they proved the following theorem, which shows how to decompose a critical graph having a separating vertex set of size two into smaller critical graphs.

\begin{theorem} [Dirac 1952 and Gallai 1963]
Let $G$ be a critical graph with chromatic number $k+1$ for an integer $k\geq 3$, and let $S\subseteq V(G)$ be a separating vertex set of $G$ with $|S|\leq 2$. Then $S$ is an independent vertex set of $G$ consisting of two vertices, say $v$ and $w$, and $G-S$ has exactly two components $H_1$ and $H_2$. Moreover, if $G_i=G[V(H_i) \cup S]$ for $i=1,2$, we can adjust the notation so that for some coloring $f_1\in \CO_k(G_1)$ we have $f_1(v)=f_1(w)$. Then the following statements hold:
\begin{itemize}
\item[{\rm (a)}] Every coloring $f\in \CO_k(G_1)$ satisfies $f(v)=f(w)$ and every coloring $f\in \CO_k(G_2)$ satisfies $f(v)\not= f(w)$.
\item[{\rm (b)}] The subgraph $G_1'=G_1+vw$ obtained from $G_1$ by adding the edge $vw$ is critical and has chromatic number $k+1$.
\item[{\rm (c)}] The vertices $v$ and $w$ have no common neighbor in $G_2$ and the subgraph $G_2'=G_2/S$ obtained from $G_2$ by identifying $v$ and $w$ is critical and has chromatic number $k+1$.
\end{itemize}
\label{Th:2Conn}
\end{theorem}

Dirac \cite{Dirac64} and Gallai \cite{Gallai63a} also proved the converse theorem, that $G$ is critical and has chromatic number $k+1$ provided that $G_1'$ is critical and has chromatic number $k+1$ and $G_2$ obtained from the critical graph $G_2'$ with chromatic number $k+1$ by splitting a vertex  into $v$ and $w$ has chromatic number $k$.

Haj\'os \cite{Hajos61} invented his construction to characterize the class of graphs with chromatic number at least $k+1$. Another advantage of the Haj\'os join is the well known fact that it not only preserve the chromatic number, but also criticality. It may be viewed as a special case of the Dirac--Gallai construction, described above.

\begin{theorem} [Haj\'os 1961]
Let $G=G_1 \bigtriangleup G_2$ be the Haj\'os join of two graphs $G_1$ and $G_2$, and let $k\geq 3$ be an integer. Then $G$ is critical and has chromatic number $k+1$ if and only if both $G_1$ and $G_2$ are critical and have chromatic number $k+1$.
\label{Th:Hajos}
\end{theorem}

If $G$ is the Haj\'os join of two graphs that are critical and have chromatic number $k+1$, where $k\geq 3$, then $G$ is critical and has chromatic number $k+1$. Moreover, $G$ has a separating set consisting of one edge and one vertex. Theorem~\ref{Th:2Conn} implies that the converse statement also holds.

\begin{theorem}
Let $G$ be a critical graph graph with chromatic number $k+1$ for an integer $k\geq 3$. If $G$ has a separating set consisting of one edge and one vertex, then $G$ is the Haj\'os join of two graphs.
\label{Th:2Sep=Hajos}
\end{theorem}

Next we will discuss a  decomposition result for critical graphs having chromatic number $k+1$ an having an separating edge set of size $k$.
Let $G$ be an arbitrary graph.
By an {\em edge cut} of $G$ we mean a triple $(X,Y,F)$ such that $X$ is a non-empty proper subset of $V(G)$, $Y=V(G)\sm X$, and $F=\bG(X)=\bG(Y)$. If $(X,Y,F)$ is an edge cut of $G$, then we denote by $X_F$ (respectively $Y_F$) the set of vertices of $X$ (respectively, $Y$) which are incident to some edge of $F$. An edge cut $(X,Y,F)$ of $G$ is non-trivial if $|X_F|\geq 2$ and $|Y_F|\geq 2$. The following decomposition result was proved independently by T. Gallai and Toft \cite{Toft70}.

\begin{theorem} [Toft 1970]
Let $G$ be a critical graph with chromatic number $k+1$ for an integer $k\geq 3$, and let $F\subseteq E(G)$ be a separating edge set of $G$ with $|F|\leq k$. Then $|F|=k$ and there is an edge cut $(X,Y,F)$ of $G$ satisfying the following properties:
\begin{itemize}
\item[{\rm (a)}] Every coloring $f\in \CO_k(G[X])$ satisfies $|f(X_F)|=1$ and  every coloring $f\in \CO_k(G[Y])$ satisfies $|f(Y_F)|=k$.
\item[{\rm (b)}] The subgraph $G_1$ obtained from $G[X \cup Y_F]$ by adding all edges between the vertices of $Y_F$, so that $Y_F$ becomes a clique of $G_1$, is critical and has chromatic number $k+1$.
\item[{\rm (c)}] The subgraph $G_2$ obtained from $G[Y]$ by adding a new vertex $v$ and joining $v$ to all vertices of $Y_F$
    is critical and has chromatic number $k+1$.
\end{itemize}
\label{Th:Toft}
\end{theorem}

A particular nice proof of this result is due to T. Gallai (oral communication to the second author). Recall that the {\em clique number} of a graph $G$, denoted by $\om(G)$, is the largest cardinality of a clique in $G$. A graph $G$ is {\em perfect} if every induced subgraph $H$ of $G$ satisfies $\cn(H)=\om(H)$. For the proof of the next lemma, due to Gallai, we use the fact that complements of bipartite graphs are perfect.

\begin{lemma}
Let $H$ be a graph and let $k\geq 3$ be an integer. Suppose that $(A,B,F')$ is an edge cut of $H$ such that $|F'|\leq k$ and $A$ as well as $B$ are cliques of $H$ with $|A|=|B|=k$. If $\cn(H)\geq k+1$, then $|F'|=k$ and $F'=\bH(\{v\})$ for some vertex $v$ of $H$.
\label{Le:perfect}
\end{lemma}
\begin{proof}
The graph $H$ is perfect and so $\om(H)=\cn(H)\geq k+1$. Consequently, $H$ contains a clique $X$ with $|X|=k+1$. Let $s=|A\cap X|$ and hence $k+1-s=|B\cap X|$. Since $|A|=|B|=k$, this implies that $s\geq 1$ and $k+1-s\geq 1$. Since $X$ is a clique of $H$, the set $E'$ of edges of $H$ joining a vertex of $A\cap X$ with a vertex of $B\cap X$ satisfies $E'\subseteq F'$ and $|E'|=s(k+1-s)$. Clearly, $g''(s)=-2$, which implies that the function $g(s)=s(k+1-s)$ is strictly concave on the real interval $[1,k]$. Since $g(1)=g(k)=k$, we conclude that $g(s)>k$ for all $s\in (1,k)$. Since $g(s)=|E'|\leq |F'|\leq k$, this implies that $s=1$ or $s=k$. In both cases we obtain that $|E'|=|F'|=k$, and hence $E'=F'=\bH(\{v\})$ for some vertex $v$ of $H$.
\end{proof}

Based on Lemma~\ref{Le:perfect} it is easy to give a proof of Theorem~\ref{Th:Toft}, see also the paper by Dirac, S{\o}rensen, and Toft \cite{DiracT74}. Theorem~\ref{Th:Toft} is a reformulation of a result by Toft in his Ph.D thesis. Toft gave a complete characterization of the class of critical graphs, having chromatic number $k+1$ and containing a separating edge set of size $k$. The characterization involves critical hypergraphs.

Figure~\ref{Fig:A1} shows three critical graphs with $\cn=4$. The first graph is an odd wheel and the second graph is the Haj\'os join of two $K_4$'s; both graphs belong to the class $\cC_3$. The third graph does not belong to $\cC_3$; it has an separating edge set of size 3, but $\la=4$.

\begin{figure}[htbp]
\centering
\includegraphics[height=3cm]{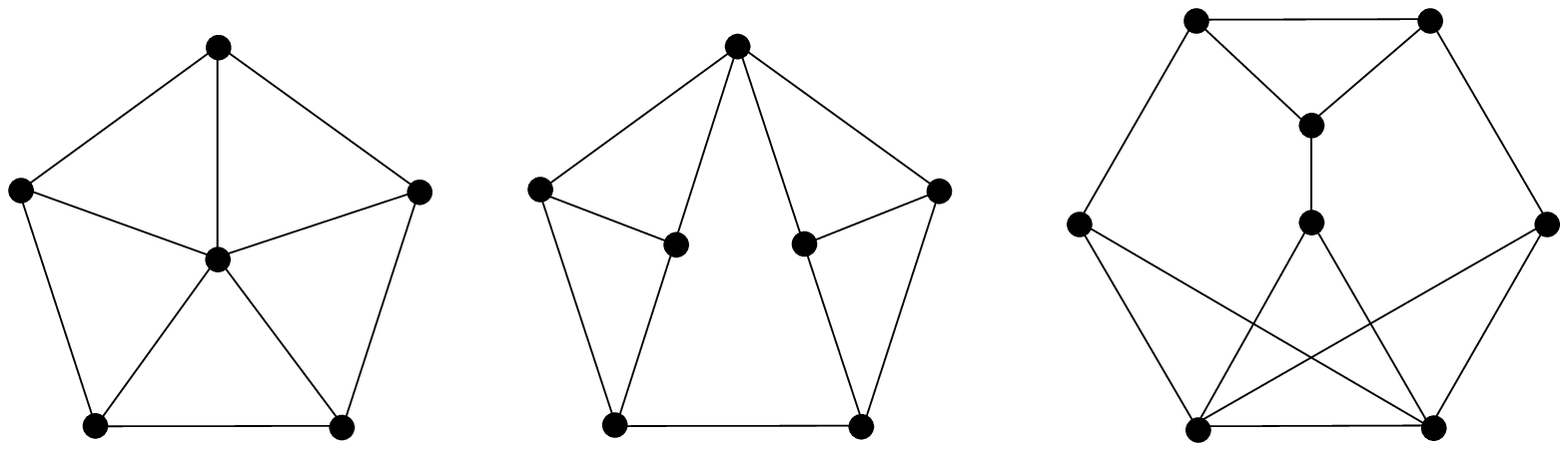}
%\picplace{5cm}{2cm} % Give the correct figure height and width in cm
%
\caption{Three critical graphs with chromatic number $\cn=4$.}
\label{Fig:A1}       % Give a unique label
\end{figure}

\section{Proof of the main result}

\begin{theorem}
Let $k\geq 0$ be an integer. Then the two graph classes $\cC_k$ and $\cH_k$ coincide.
\label{Th:Ck=Hk}
\end{theorem}
\begin{proof}
That the two classes $\cC_k$ and $\cH_k$ coincide if $0\leq k \leq 2$ follows from Theorem~\ref{Th:Dirac}(a). In this case both classes consists of all critical graphs with chromatic number $k+1$. In what follows we therefore assume that $k\geq 3$. The proof of the following claim is straightforward and left to the reader.

\begin{claim}
The odd wheels belong to the class $\cC_3$ and the complete graphs of order $k+1$ belong to the class $\cC_k$.
\label{Cl:A1}
\end{claim}

\begin{claim}
Let $k\geq 3$ be an integer, and let $G=G_1 \bigtriangleup G_2$ the Haj\'os join of two graphs $G_1$ and $G_2$. Then $G$ belongs to the class $\cC_k$ if and only if both $G_1, G_2$ belong to the class $\cC_k$.
\label{Cl:A2}
\end{claim}
\pfcl{We may assume that $G=(G_1,v_1,w_1) \bigtriangleup (G_2,v_2,w_2)$ and $v$ is the vertex of $G$ obtained by identifying $v_1$ and $v_2$. First suppose that $G_1, G_2\in \cC_k$. From Theorem~\ref{Th:Hajos} it follows that $G$ is critical and has chromatic number $k+1$. So it suffices to prove that $\la(G)\leq k$. To this end let $u$ and $u'$ be distinct vertices of $G$ and let $p=\la_G(u,u')$. Then there is a system $\cP$ of $p$ edge disjoint $u$-$u'$ paths in $G$. If $u$ and $u'$ belong both to $G_1$, then only one path $P$ of $\cP$ may contain vertices not in $G_1$. In this case $P$ contains the vertex $v$ and the edge $w_1w_2$. If we replace in $P$ the subpath $vPw_1$ by the edge $v_1w_1$, we obtain a system of $p$ edge disjoint $u$-$u'$ paths in $G_1$, and hence $p\leq \la_{G_1}(u,u')\leq k$. If $u$ and $u'$ belong to $G_2$, a similar argument shows that $p\leq k$. It remains to consider the case that one vertex, say $u$, belongs to $G_1$ and the other vertex $u'$ belongs to $G_2$. By symmetry we may assume that $u\not=v$. Again at most one path $P$ of $\cP$ uses the edge $w_1w_2$ and the remaining paths of $\cP$ all uses the vertex $v(=v_1=v_2)$. If we replace $P$ by the path $uPw_1+w_1v_1$, then we obtain $p$ edge disjoint $u$-$v_1$ path in $G_1$, and hence $p\leq \la_{G_1}(u,v_1)\leq k$. This shows that $\la(G)\leq k$ and so $G\in \cC_k$.

Suppose conversely that $G\in \cC_k$. From Theorem~\ref{Th:Hajos} it follows that $G_1$ and $G_1$ are critical graphs, both with chromatic number $k+1$. So it suffices to show that $\la(G_i)\leq k$ for $i=1,2$. By symmetry it suffices to show that $\la(G_1)\leq k$. To this end let $u$ and $u'$ be distinct vertices of $G_1$ and let $p=\la_G(u,u')$. Then there is a system $\cP$ of $p$ edge disjoint $u$-$u'$ paths in $G_1$. At most one path $P$ of $\cP$ can contain the edge $v_1w_1$. Clearly, there is a $v_2$-$w_2$ path $P'$ in $G_2$ not containing the edge $v_2w_2$. So if we replace the edge $v_1w_1$ of $P$ by the path $P'$, we get $p$ edge disjoint $u$-$u'$ paths of $G$, and hence $p\leq \la_G(u,u')\leq k$. This shows that $\la(G_1)\leq k$ and by symmetry $\la(G_2)\leq k$. Hence $G_1, G_2\in \cC_k$.
}

As a consequence of Claim~\ref{Cl:A1} and Claim~\ref{Cl:A2} and the definition of the class $\cH_k$ we obtain the following claim.

\begin{claim}
Let $k\geq 3$ be an integer. Then the class $\cH_k$ is a subclass of $\cC_k$.
\label{Cl:A3}
\end{claim}

\begin{claim}
Let $k\geq 3$ be an integer, and let $G$ be a graph belonging to the class $\cC_k$. If $G$ is 3-connected, then either $k=3$ and $G$ is an odd wheel, or $k\geq 4$ and $G$ is a complete graph of order $k+1$.
\label{Cl:A4}
\end{claim}
\pfcl{The proof is by contradiction, where we consider a counterexample $G$ whose order $|G|$ is minimum. Then $G\in \cC_k$ is a 3-connected graph, and either $k=3$ and $G$ is not an odd wheel, or $k\geq 4$ and $G$ is not a complete graph of order $k+1$. First we claim that $|G_H|\geq 2$. If $G_H=\ems$, then Theorem~\ref{Th:Gallai}(b) implies that $G$ is a complete graph of order $k+1$, a contradiction. If $|G_H|=1$, then Theorem~\ref{Th:Gallai}(c) implies that $k=3$ and $G$ is an odd wheel, a contradiction. This proves the claim that $|G_H|\geq 2$. Then let $u$ and $v$ be distinct high vertices of $G$. Since $G\in \cC_k$, Theorem~\ref{Th:Dirac}(d) implies that $\la_G(u,v)=k$ and, therefore, $G$ contains a separating edge set $F$ of size $k$ which separates $u$ and $v$. From Theorem~\ref{Th:Toft} it then follows that there is an edge cut $(X,Y,F)$ satisfying the three properties of that theorem. Since $F$ separates $u$ and $v$, we may assume that $u\in X$ and $v\in Y$. By Theorem\ref{Th:Toft}(a), $|Y_F|=k$ and hence each vertex of $Y_F$ is incident to exactly one edge of $F$. Since $Y$ contains the high vertex $v$, we conclude that $|Y_F|<|Y|$. Now we consider the graph $G'$ obtained from $G[X \cup Y_F]$ by adding all edges between the vertices of $Y_F$, so that $Y_F$ becomes a clique of $G'$. By Theorem~\ref{Th:Toft}(b), $G'$ is a critical graph with chromatic number $k+1$. Clearly, every vertex of $Y_F$ is a low vertex of $G$ and every vertex of $X$ has in $G'$ the same degree as in $G$. Since $X$ contains the high vertex $u$ of $G$, this implies that $|X_F|<|X|$. Since $G$ is 3-connected, we conclude that $|X_F|\geq 3$ and that $G'$ is 3-connected.

Now we claim that $\la(G')\leq k$. To prove this, let $x$ and $y$ be distinct vertices of $G'$. If $x$ or $y$ is a low vertex of $G'$, then $\la_{G'}(x,y)\leq k$ and there is nothing to prove. So assume that both $x$ and $y$ are high vertices of $G'$. Then both vertices $x$ and $y$ belong to $X$. Let $p=\la_{G'}(x,y)$ and let $\cP$ be a system of $p$ edge disjoint $x$-$y$ paths in $G'$. We may choose $\cP$ such that the number of edges in $\cP$ is minimum.
Let $\cP_1$ be the paths in $\cP$ which uses edges of $F$. Since $|Y_F|=k$ and each vertex of $Y_F$ is incident with exactly one edge of $F$, this implies that each path $P$ in $\cP_1$ contains exactly two edges of $F$. Since $|X_F|<|X|$ and $|Y_F|<|Y|$, there are vertices $u'\in X\sm X_F$ and $v'\in Y\sm Y_F$. By Theorem~\ref{Th:Dirac}(d) it follows that $\la_G(u',v')=k$ and, therefore, there are $k$ edge disjoint $u'$-$v'$ paths in $G$. Since $|Y_F|=k$, for each vertex $z\in Y_F$, there is a $v'$-$z$ path $P_z$ in $G[Y]$ such that these paths are edge disjoint. Now let $P$ be an arbitrary path in $\cP_1$. Then $P$ contains exactly two vertices of $Y_F$, say $z$ and $z'$, and we can replace the edge
$zz'$ of the path $P$ by a $z$-$z'$ path contained in $P_z \cup P_{z'}$. In this way we obtain a system of $p$ edge disjoint $x$-$y$ paths in $G$, which implies that $p\leq \la_G(x,y)\leq k$. This proves the claim that $\la(G')\leq k$. Consequently $G'\in \cC_k$. Clearly, $|G'|<|G|$ and either $k=3$ and $G'$ is not an odd wheel, or $k\geq 4$ and $G$ is not a complete graph of order $k+1$. This, however, is a contradiction to the choice of $G$. Thus the claim is proved.
}

\begin{claim}
Let $k\geq 3$ be an integer, and let $G$ be a graph belonging to the class $\cC_k$. If $G$ has a separating vertex set of size 2, then $G=G_1\bigtriangleup G_2$ is the Haj\'os sum of two graphs $G_1$ and $G_2$, which both belong to $\cC_k$.
\label{Cl:A5}
\end{claim}
\pfcl{If $G$ has a separating set consisting of one edge and one vertex, then Theorem~\ref{Th:2Sep=Hajos} implies that $G$ is the Hajo\'s join of two graphs $G_1$ and $G_2$. By Claim~\ref{Cl:A2} it then follows that both $G_1$ and $G_2$
belong to $\cC_k$ and we are done. It remains to consider the case that $G$ does not contain a separating set consisting of one edge and one vertex. By assumption, there is a separating vertex set of size 2, say $S=\{u,v\}$. Then Theorem~\ref{Th:2Conn} implies that $G-S$ has exactly two components $H_1$ and $H_2$ such that the graphs $G_i=G[V(H_i) \cup S]$ with $i=1,2$ satisfies the three properties of that theorem. In particular, we have that $G_1'=G_1+uv$ is critical and has chromatic number $k$. By Theorem~\ref{Th:Dirac}(c), it then follows that $\la_{G_1'}(u,v)\geq k$ implying that $\la_{G_1}(u,v)\geq k-1$. Since $G\in \cC_k$, we then conclude that $\la_{G_2}(u,v)\leq 1$. Since $G_2$ is connected, this implies that $G_2$ has a bridge $e$. Since $k\geq 3$, we conclude that $\{u,e\}$ or $\{v,e\}$ is a separating set of $G$, a contradiction.
}

As a consequence of Claim~\ref{Cl:A4} and Claim~\ref{Cl:A5}, we conclude that the class $\cC_k$ is a subclass of the class $\cH_k$. Together with Claim~\ref{Cl:A3} this yields $\cH_k=\cC_k$ as wanted.
\end{proof}

\pff{of Theorem~\ref{Th:local}}{For the proof of this theorem let $G$ be a non-empty graph with $\la(G)=k$. By (\ref{Equ:lambda}) we obtain that $\cn(G)\leq k+1$. If one block $H$ of $G$ belongs to $\cH_k$, then $H\in \cC_k$ (by Theorem~\ref{Th:Ck=Hk}) and hence $\cn(G)=k+1$ (by (\ref{Equ:Block}).

Assume conversely that $\cn(G)=k+1$. Then $G$ contains a subgraph $H$ which is critical and has chromatic number $k+1$. Clearly, $\la(H)\leq \la(G)\leq k$, and, therefore, $H\in \cC_k$. By Theorem~\ref{Th:Dirac}(b), $H$ contains no separating vertex. We claim that $H$ is a block of $G$. For otherwise, $H$ would be a proper subgraph of a block $G'$ of $G$. This implies that there are distinct vertices $u$ and $v$ in $H$ which are joined by a path $P$ of $G$ with $E(P)\cap E(H)=\ems$. Since $\la_H(u,v)\geq k$ (by Theorem~\ref{Th:Dirac}(c)), this implies that $\la_G(u,v)\geq k+1$, which is impossible. This proves the claim that $H$ is a block of $G$. By Theorem~\ref{Th:Ck=Hk}, $\cC_k=\cH_k$ implying that $H\in \cH_k$. This completes the proof of the theorem}

\medskip

The case $\la=3$ of Theorem~\ref{Th:local} was obtained earlier by Alboulker {\em et al.\,} \cite{AlboukerV2016}; their proof is similar to our proof. Let $\cL_k$ denote the class of graphs $G$ satisfying $\la(G)\leq k$. It is well known that membership in $\cL_k$ can be tested in polynomial time. It is also easy to show that there is a polynomial-time algorithm that, given a graph $G\in \cL_k$, decides whether $G$ or one of its blocks belong to $\cH_k$. So it can be tested in polynomial time whether a graph $G\in \cL_k$ satisfies $\cn(G)\leq k$.
Moreover, the proof of Theorem~\ref{Th:local} yields a polynomial-time algorithm that, given a graph $G\in \cL_k$, finds a coloring of $\CO_k(G)$ when such a coloring exists.
This result provides a positive answer to a conjecture made by Alboulker {\em et al.\,} \cite[Conjecture 1.8]{AlboukerV2016}. The case $k=3$ was solved by Alboulker {\em et al.\,} \cite{AlboukerV2016}.

\begin{theorem}
For fixed $k\geq 1$, there is a polynomial-time algorithm that, given a graph $G\in \cL_k$, finds a coloring in $\CO_k(G)$ or a block belonging to $\cH_k$.
\label{Th:Algorithm}
\end{theorem}
\skpf{The Theorem is evident if $k=1,2$; and the case $k=3$ was solved by Alboulker {\em et al.\,} \cite{AlboukerV2016}. Hence we assume that $k\geq 4$ and $G\in \cL_k$. If we find for each block $H$ of $G$ a coloring in $\CO_k(H)$, we can piece these colorings together by permuting colors to obtain a coloring in $\CO_k(G)$. Hence we may assume that $G$ is a block. First, we check whether $G$ has a separating set $S$ consisting of one vertex and one edge. If we find such a set, say $S=\{v,e\}$ with $v\in V(G)$ and $e\in E(G)$. Then $G-e$ is the union of two connected graphs $G_1$ and $G_2$ having only vertex $v$ in common where $e=w_1w_2$ and $w_i\in V(G_i)$ for $i=1,2$. Both blocks $G_1'=G_1+vw_1$ and $G_2'=G_2+vw_2$ belong to $\cL_k$. Now we check whether these blocks belong to $\cH_k$. If both blocks $G_1'$ and $G_2'$ belong to $\cH_k$, then $vw_i\not\in E(G_i)$ for $i=1,2$, and hence $G$ belongs to $\cH_k$ and we are done. If one of the blocks, say $G_1'$ does not belong to $\cH_k$, we can construct a coloring $f_1\in \CO_k(G_1')$. Moreover, no block of $G_2$ belongs to $\cH_k$, hence we can construct a coloring $f_2\in \CO_k(G_2)$. Then $f_1\in \CO_k(G_1)$ and $f_1(v)\not=f_1(w_1)$. Since $k\geq 4$, we can permute colors in $f_2$ such that $f_1(v)=f_2(v)$ and $f_1(w_1)\not=f_2(w_2)$. Consequently, $f=f_1 \cup f_2$ belongs to $\CO_k(G)$ and we are done.

It remains to consider the case that $G$ contains no separating set consisting of one vertex and one edge. Then let $p$ denote the number of vertices of $G$ whose degree is greater that $k$. If $p\leq 1$, then let $v$ be a vertex of maximum degree in $G$. Color $v$ with color $1$ and let $L$ be a list assignment for $H=G-v$
satisfying $L(u)=\{2,3, \ldots ,k\}$ if $vu\in E(G)$ and $L(u)=\{1,2, \ldots, k\}$ otherwise. Then $H$ is connected and $|L(u)|\geq d_H(u)$ for all $u\in V(H)$. Now we can use the degree version of Brooks' theorem, see \cite[Theorem 2.1]{StiebitzT2015}. Either we find a coloring $f$ of $H$ such that $f(u)\in L(u)$ for all $u\in V(H)$, yielding a coloring of $\CO_k(G)$, or $|L(u)|=d_H(u)$ for all $u\in V(H)$ and each block of $H$ is a complete graph or an odd cycle. In this case, $d_H(u)\in \{k,k-1\}$ for all $u\in V(H)$ and, since $k\geq 4$, each block of $H$ is a $K_k$ or a $K_2$. Since $G$ contains no separating set consisting of one vertex and one edge, this implies that $H=K_k$ and so $G=K_{k+1}\in \cH_k$ and we are done. If $p\geq 2$, then we choose two vertices $u$ and $u'$ whose degrees are greater that $k$. Then we construct an edge cut $(X,Y,F)$ with $u\in X$, $u'\in Y$, and $|F|=\la_G(u,u')$. We may assume that $a=|X_F|$ and $b=|Y_F|$ satisfies $a\leq b\leq k$. If $b\leq k-1$, then both graphs $G[X]$ and $G[Y]$ belong to $\cL_k$ and there are colorings $f_X\in \CO_k(G[X])$ and $f_Y\in \CO_k(G[Y])$. Note that no block of these two graphs can belong to $\cH_k$. By permuting colors in $f_Y$, we can combine the two colorings $f_X$ and $f_Y$ to obtain a coloring $f\in \CO_k(G)$ (by Lemma~\ref{Le:perfect}). If $a<b=k$, then we consider the graph $G_1$ obtained from $G[X \cup Y_F]$ by adding all edges between the vertices of $Y_F$, so that $Y_F$ becomes a clique of $G_1$. Then $G_1$ belongs to $\cL_k$ (see the proof of Claim 4) and, since $G$ contains no separating set consisting of one vertex and one edge, the block $G_1$ does not belongs to $\cH_k$. Hence there are colorings $f_1\in \CO_k(G_1)$ and $f_Y\in \CO_k(G[Y])$. Then the restriction of $f_1$ to $X$ yields a coloring $f_X\in \CO_k(G[X])$ such that $|f_X(X)|\geq 2$. By permuting colors in $f_Y$, we can combine the two colorings $f_X$ and $f_Y$ to obtain a coloring $f\in \CO_k(G)$ (by Lemma~\ref{Le:perfect}). It remains to consider the case $a=b=k$. Then let $G_2$ be the graph obtained from $G[Y \cup X_F]$ by adding all edges between the vertices of $X_F$, so that $X_F$ becomes a clique of $G_2$. Then we find colorings $f_1\in \CO_k(G_1)$ and $f_2\in \CO_k(G_2)$ and, hence, colorings $f_X\in \CO_k(G[X])$ and $f_Y\in \CO_k(G[Y])$ such that $|f_X(X)|\geq 2$ and $|f_Y(Y)|\geq 2$.  By permuting colors in $f_Y$, we can combine the two colorings $f_X$ and $f_Y$ to obtain a coloring $f\in \CO_k(G)$ (by Lemma~\ref{Le:perfect}).}

\end{document}